\documentclass{amsart}
\usepackage{amsmath}
\usepackage{amsfonts}
\usepackage{amssymb}

\newtheorem{theorem}{Theorem}[section]
\newtheorem{lemma}[theorem]{Lemma}
\newtheorem{corollary}[theorem]{Corollary}

\newtheorem*{theorem*}{Theorem}
\newtheorem*{acknowledgement*}{Acknowledgement}

\theoremstyle{definition}

\newtheorem*{claim*}{Claim}

\theoremstyle{remark}

\newcommand{\Rc}[0]{\operatorname{Rc}}
\newcommand{\Rm}[0]{R}
\newcommand{\defn}[0]{\doteqdot}

\newcommand{\pdt}[0]{\frac{\partial}{\partial t}}
\newcommand{\heatop}{\left(\pdt - \Delta\right)}
\newcommand{\supp}[0]{\operatorname{supp}}
\newcommand{\gt}[0]{\tilde{g}}

\numberwithin{equation}{section}

\begin{document}

\title[A local version of Bando's theorem]{A local version of Bando's theorem on the real-analyticity of solutions
to the Ricci flow}

\author{Brett Kotschwar}
\address{Max Planck Institut f\"ur Gravitationsphysik, Golm, Germany}
\address{Arizona State University, Tempe, Arizona, USA}
\email{kotschwar@math.asu.edu}

\thanks{The author was supported in part by NSF grant DMS-0805834}

\date{June 2011}

\begin{abstract}
  It is a theorem of S. Bando that if $g(t)$ is a solution to the Ricci flow on a compact manifold $M$,
  then $(M, g(t))$ is real-analytic for each $t >0$. In this note, we extend his result to smooth solutions on open domains 
  $U\subset M$.
\end{abstract}
\maketitle

\section{Introduction}

Suppose $U$ is an open subset of $M= M^n$. We consider a smooth solution $g(t)$ to the Ricci flow 
\begin{equation}\label{eq:rf}
  \pdt g = -2\Rc(g)
\end{equation}
on $U\times[0, T]$.  The purpose of this note is to establish the following result.
\begin{theorem}\label{thm:analyticity}
For $0 < t \leq T$, $(U, g(t))$ is a real-analytic manifold.
\end{theorem}
More precisely, about each point $p\in U$ there is a neighborhood on which the expression of $g(t)$ in geodesic normal coordinates
is real-analytic.  This is a result of S. Bando \cite{Bando} when $M$ is compact, and his argument extends, essentially without change,
to the case when $(M, g(t))$ is complete and of uniformly bounded curvature. Our aim is to eliminate the global assumptions on the metric,
and verify that, as is typical of parabolic equations, instantaneous analyticity in the spatial variables is a purely local phenomenon.

From Theorem \ref{thm:analyticity} and a classical monodromy-type argument (cf. Corollary 6.4 in \cite{KobayashiNomizu}), it is then automatic
to obtain the following qualitative unique-continuation results for complete solutions (of possibly unbounded curvature).
\begin{corollary}
 Suppose that $M$ is connected and simply-connected, and 
$g(t)$, $\gt(t)$ are complete solutions to the Ricci flow on $M\times(a, b)$.  Let $t_0 \in (a, b)$. 
  \begin{enumerate}
    \item If $g(\cdot, t_0)= \gt(\cdot, t_0)$ on an connected open set $U\subset M$, then there exists a diffeomorphism
    $\phi: M\to M$ such that $g(t_0)= \phi^{*}\gt(t_0)$,
    \item Any local isometry, $\tilde{\phi}: (U, g|_U(t_0)) \to (V, g|_V(t_0))$,
	  between connected open sets $U$, $V\subset M$ can be uniquely extended to a global isometry 
	  $\phi:M \to M$.
 \end{enumerate}
\end{corollary}

Since any local Ricci soliton may be transformed into a local (self-similar) solution
to Ricci flow, Theorem \ref{thm:analyticity} also provides a new proof of the real-analyticity of Ricci solitons, 
a fact which can be proven, much as for Einstein metrics (cf. \cite{DeTurckKazdan}, \cite{Ivey}), by the use of harmonic coordinates.
\begin{corollary}
  Suppose $(M, g, X, \lambda)$ is a Ricci soliton, i.e.,
\[ 
\Rc(g) + \mathcal{L}_X g + \lambda g = 0,
\]
 for some smooth vector field $X$ and scalar $\lambda$. 
  Then $(M, g)$ is a real-analytic manifold.
\end{corollary}
Note that since $\Delta X = - \Rc(X)$ on a Ricci soliton, it follows that the representation of the vector field 
$X$ in geodesic normal coordinates will also be analytic.

Theorem \ref{thm:analyticity} will be a consequence of the following estimate.  Here and below, $R$ denotes the Riemann curvature tensor,
and $\Omega(p, r)$ and $\Omega(p, r, T)$ denote, respectively, $B_{g(0)}(p, r)$ and $\Omega(p, r)\times [0, T]$.
\begin{theorem}\label{thm:localest}
Suppose $g(x, t)$ is a smooth solution to the Ricci flow on
the open set $U\subset M^n$ for $t\in [0, T]$.  Let $p \in U$ and $\rho > 0$ such that $\overline{\Omega(p, 3\rho)}$ is compactly contained
in $U$ and define $M_0 \defn \sup_{\Omega(p, 3\rho, T)}|R|_{g(t)}$.
Then, there exist positive constants $C$, $N$, and $\tau$ depending only on $n$, $\rho$, $T$ and $M_0$ such that for all $m\in \mathbb{N}\cup\{0\}$,
\begin{equation}\label{eq:localest}
 t^{m/2}|\nabla^{m}R|_{g(t)}(x,t) \leq CN^{m/2}(m+1)!
\end{equation}
on $\Omega(p, \rho, \tau)$.
\end{theorem}

The estimates \eqref{eq:localest} are variants of the well-known local
estimates of Shi \cite{Shi} and Hamilton \cite{HamiltonSingularities}  (see also \cite{LuTian}, \cite{ShermanWeinkove}), which likewise take the form
\[
t^m|\nabla^{m}R|^2\leq C(n, m, \rho, M_0, T).
\] 
The only new content is an explicit accounting of the dependency of the constants $C(n, m, \rho, M_0, T)$ on 
the order $m$ -- a dependency that is often unimportant in applications and consequently rather obscure in the variants of the 
estimates of which we are aware.  In fact, it is an interesting question whether, for example, the constants
generated inductively in Shi's argument are of sufficiently slow growth in $m$ to ensure the real-analyticity.
The application of the 
heat operator to the quantity $|\nabla^{m}R|^2$, the subsequent commutation of the Laplacian with the $m$-fold covariant derivative,
and the $m$-fold differentiation of the reaction terms on the right-hand side of $(\pdt - \Delta) R = R\ast R$
together generate a number of lower order terms which grow with $m$.
To control these terms, we found it easiest to use a localized modification of Bando's original quantity
\[
    \varphi = \sum_{k=0}^m \frac{t^k}{((k+1)!)^2}|\nabla^k R|^2,
\]
whose evolution equation can be arranged, as in the global case,
produce a comparison with the solution of an appropriate ODE.

\section{Proof of the local estimates}

For the remainder of this paper, we will work in the setting of the statement of Theorem \ref{thm:localest}.
We first describe our cut-off function.  

\begin{lemma}\label{lem:cutoff}
  Under the assumptions of Theorem \ref{thm:localest}, there exists a cut-off function $\eta: U \to [0, 1]$ that is compactly
  supported in $\Omega(p, 2\rho)$, satisfies $\eta \equiv 1$ on $\Omega(p, \rho)$, and whose derivatives satisfy 
  \begin{align}\label{eq:cutoffderivative}
      |\nabla\eta|^2_{g(t)} -\eta\Delta_{g(t)}\eta \leq C_0\eta 
  \end{align}
  on $U\times [0, T]$ for some constant $C_0 = C_0(n, \rho, T, M_0)$.
\end{lemma}
\begin{proof}
 See, e.g, Lemma 14.4 of \cite{RFV2P2}.  The key is that the local curvature bound implies the uniform equivalence of the metrics $g(t)$
 on $\overline{\Omega}(p, 3\rho, T)$ and the local first derivative estimate of either Shi \cite{Shi} or Hamilton \cite{HamiltonSingularities}
 supplies a bound of the form
 \[
      \sup_{\overline{\Omega}(p, 2\rho, T)} t|\nabla \Rm|_{g(t)}^2(x, t) \leq C(n, \rho, T, M_0),
 \]
 and these, together, are sufficient to control the $g(t)$-gradient and Laplacian of the cut-off function. 
\end{proof}

Although $\eta$ is constructed by composition with a Riemannian distance function -- namely, that of the initial metric $g(0)$ -- 
we may, as usual, on account of Calabi's trick \cite{Calabi}, regard the resulting function as smooth for the purpose of applying the maximum principle.  
(Alternatively, at the outset
of what follows, we may simply decrease $\rho$ if necessary to ensure that $3\rho < \operatorname{inj}_p(g(0))$.)

Now we introduce the principal quantity in our estimate, a simple modification
of that introduced by Bando \cite{Bando}. We define 
\[
  A_k \defn \left(\frac{t}{N}\right)^{k/2}\frac{|\nabla^{k}\Rm|}{(k+1)!}, \quad \phi_k \defn \eta^{k+1}A^2_k,
\]
for $k = 0, 1, 2, \ldots$ and
\[
   B_k \defn \left(\frac{t}{N}\right)^{(k-1)/2}\frac{|\nabla^k\Rm|}{k!},\quad
 \psi_k \defn \eta^{k}B^2_k, 
\]
for $k = 1, 2, 3, \ldots$. We then compute
\begin{equation}\label{eq:phikevol}
 \heatop \phi_k = \eta^{k+1}\heatop A^2_k + A^2_k\heatop \eta^{k+1} - 2\langle\nabla\eta^{k+1}, \nabla A^2_k\rangle.
\end{equation}

We will split our computations into two cases, depending as $k \geq 1$ or $k = 0$. We consider first the case 
$k\geq 1$. As in \cite{Bando}, the quantity $A^2_k$ can be seen to satisfy
\begin{equation}\label{eq:firstterm}
  \heatop A^2_k \leq -2B^2_{k+1} +\frac{k}{N(k+1)^2}B^2_k + C_1 M_0 A^2_k
  + \frac{C_1t}{N} S_k
\end{equation}
where $C_1 = C_1(n)$ is independent of $k$ and
\[
  S_k \defn \sum_{i=0}^{k-1}\frac{A_i B_{k-i} B_{k}}{i+2}.
\]

Using \eqref{eq:firstterm} and $0 \leq \eta \leq 1$, the first term of \eqref{eq:phikevol} thus satisfies
\begin{align*}
\eta^{k+1}\heatop A^2_k &\leq -2\psi_{k+1} + \frac{1}{N}\psi_k + C_1M_0\phi_k + \frac{C_1t}{N}\theta_k
\end{align*}
for  $t \leq T$, where
\[
  \theta_k \defn \eta^{k+1/2}S_k = \sum_{i=0}^{k-1}\frac{\phi^{1/2}_i\psi_{k-i}^{1/2}\psi_{k}^{1/2}}{i+2}.
\]

Then, since
\begin{align*}
 -\Delta \eta^{k+1} &= - k(k+1)\eta^{k-1}|\nabla \eta|^2 - (k+1)\eta^k\Delta \eta\\
		    &\leq (k+1)C_0\eta^k,
\end{align*}
and $(k+1)A_k = \sqrt{t/N}B_k$,
the second term of \eqref{eq:phikevol} satisfies
\begin{align}\label{eq:secondterm}
 A_k^2\heatop \eta^{k+1} &\leq (k+1)C_0\eta^kA_k^2 \leq \frac{C_0T}{N(k+1)}\psi_k.
\end{align}

Finally, on $\supp \eta$, the last term of \eqref{eq:phikevol} may be estimated by
\begin{align}\label{eq:lastterm}\begin{split}
  &-2\langle \nabla\eta^{k+1}, \nabla A^2_k\rangle
	\leq 4\frac{(k+1)}{((k+1)!)^2}\left(\frac{t}{N}\right)^k\eta^k|\nabla\eta||\nabla^{k+1}\Rm||\nabla^k\Rm|\\
  &\qquad\quad\leq 4 \frac{|\nabla\eta|}{\eta^{1/2}}\sqrt{\frac{t}{N}}\left(t^{k/2}\eta^{(k+1)/2}\frac{|\nabla^{k+1}R|}{(k+1)!}\right)
	      \left(t^{(k-1)/2}\eta^{k/2}\frac{|\nabla^{k}R|}{k!}\right)\\
  &\qquad\quad\leq \psi_{k+1} + \frac{4C_0t}{N}\psi_k.
\end{split}
\end{align}

Taken together, \eqref{eq:firstterm}, \eqref{eq:secondterm}, and \eqref{eq:lastterm} imply that, for $k \geq 1$,
\begin{equation}\label{eq:phik}
 \heatop \phi_k \leq -\psi_{k+1} + \frac{C_2}{N}\psi_k + C_1M_0\phi_k + \frac{C_1t}{N} \theta_k
\end{equation}
where $C_2 = C_2(n, M_0, \rho, T)$. In the case $k = 0$, we may estimate $\phi_0 = \eta|\Rm|^2$ in the same way, obtaining
\begin{equation}\label{eq:phi0}
  \heatop \phi_0 \leq -\psi_1 + C_1 M_0 \phi_0 + C_3
\end{equation}
for some constant $C_3= C_3(n, M_0, \rho, T)$.

Now we define
\[
  \Phi_m \defn \sum_{k=0}^m\phi_m, \quad \Psi_m\defn \sum_{k=1}^{m}\psi_k, \quad \Theta_m \defn \sum_{k=1}^{m}\theta_k. 
\]
Provided we choose $N = N(n, M_0, \rho, T)$ suitably large ($N >  2 \times \operatorname{max}{(C_1, C_2)}$ is sufficient),
equations \eqref{eq:phik} and \eqref{eq:phi0} combine to produce the estimate
\begin{equation}\label{eq:Phimevol1}
  \heatop \Phi_m \leq -\frac{1}{2}(\Psi_m -t\Theta_m) + C_4(\Phi_m +  1)
\end{equation}
for $C_4 = C_1M_0 + C_3$
Before we apply the maximum principle, it remains only to estimate $\Theta_m$, and this may be done just as for the corresponding quantity in \cite{Bando}
(see also Chapter 13.2 in \cite{RFV2P2}); we reproduce the estimate here for completeness:
\begin{align*}
  \Theta^2_m &=\left(\sum_{k=1}^m\sum_{i=0}^{k-1}\frac{1}{i+2}\phi^{1/2}_i\psi^{1/2}_{k-i}\psi^{1/2}_{k}\right)^2\\
	     &\leq\sum_{k=1}^m\left(\sum_{i=0}^{k-1}\frac{1}{i+2}\phi^{1/2}_i\psi^{1/2}_{k-i}\right)^2\sum_{k=1}^m\psi_k\\
	     &\leq\sum_{k=1}^m\left\{\left(\sum_{i=0}^{k-1}\frac{1}{(i+2)^2}\right)\left(\sum_{i=0}^{k-1}\phi_i\psi_{k-i}\right)\right\}\Psi_m\\
	     &\leq\Phi_m \Psi_m^2,
\end{align*}
where we have used that $\sum_{i=0}^{\infty}1/(i+2)^2 < 1$.
So $\Theta_m \leq \Phi_m^{1/2}\Psi_m$, and returning to \eqref{eq:Phimevol1},
we have
\begin{equation}\label{eq:Phimevol2}
  \heatop \Phi_m \leq -\frac{1}{2}\Psi_m(1-t\Phi^{1/2}_m) + C_4(\Phi_m +  1).
\end{equation}

For the time being, let $\tau_m$ denote
\[
    \tau_m = \sup\{a\in [0, T] \mid t\Phi^2_m(x, t) \leq 1\quad\mbox{for all}\quad (x, t)\in U\times[0, a]\}. 
\]
We will soon show that there exists a constant $\tau = \tau(n, M_0, \rho, T) > 0$ for which $\tau_m \geq \tau$
for all $m$, but for now, simply note that, owing to the compact support of each $\Phi_m(\cdot, t)$ in $\Omega(p, 2\rho)$, we at least have
$\tau_m > 0$ for all $m$. 

Let
\[
  F(t) = (M_0^2 + 1)\exp(C_4t) - 1,
\]
so that $F$ solves $F^{\prime} = C_4(F+1)$ with $F(0) = M_0^2$.  The function $\Upsilon_m \defn \Phi_m - F$
then satisfies $\Upsilon_m \leq 0$ on the parabolic boundary of $\Omega(p, 2\rho, T)$
and
\[
  \heatop \Upsilon_m \leq 0
\]
on $\Omega(p, 2\rho, \tau_m)$.
Thus, on $\Omega(p, 2\rho, \tau_m)$ we have, by the maximum principle,
\begin{equation}\label{eq:Phimest}
    \Phi_m(x, t) \leq F(t) \leq (M_0^2 + 1)\exp(C_4 T) \defn C(n, \rho, T, M_0)
\end{equation}

But it is clear now that if $\tau$ is the lesser of $T$ and $C^{-1/2}$, then $t^2\Phi_m(x, t) \leq 1$ for $t \leq \tau$. So we have $\tau_m \geq \tau$
for any $m$.  From \eqref{eq:Phimest}, it follows in particular that, for all $m \geq 0$ and $(x, t) \in \Omega(p, \rho, \tau)$, we have
\begin{align*}
  t^m|\nabla^m \Rm|^2(x, t) = t^m\eta^{m+1}(x)|\nabla^m \Rm|^2(x,t) \leq CN^m((m+1)!)^2,
\end{align*}
which is the estimate \eqref{eq:localest}.  Hence $g(x,t)$ is real-analytic (in geodesic coordinates) at $(p, t)$ for any $0 < t \leq \tau$. 
Iterating this argument proves the same for any $t\in (0, T]$, and it follows that $(U, g(t))$ is real-analytic for any $0 < t \leq T$.


\begin{thebibliography}{A}


\bibitem[B]{Bando}  Bando, Shigetoshi.  
  ``Real analyticity of solutions of Hamilton's equation.''
    \textit{Math. Z.} 195 (1987), no. 1, 93--97. 

\bibitem[C]{Calabi} Calabi, Eugenio. 
	``An extension of E. Hopf's maximum principle with an application to Riemannian geometry.''  
	\textit{Duke Math. J.} 25 (1957) 45--56.

\bibitem[CRF]{RFV2P2}  Chow, Bennett; Chu, Sun-Chin; Glickenstein, David; Guenther, Christine; 
  Isenberg, James; Ivey, Tom; Knopf, Dan; Lu, Peng; Luo, Feng; Ni, Lei 
  \textit{The Ricci flow: techniques and applications. Part II. Analytic aspects.}
   Mathematical Surveys and Monographs, 144. American Mathematical Society, Providence, RI, 2008. xxvi+458 pp.

\bibitem[DK]{DeTurckKazdan}  DeTurck, Dennis M.; Kazdan, Jerry L.
  ``Some regularity theorems in Riemannian geometry.''
  \textit{Ann. Sci. \'{E}cole Norm. Sup.} (4) 14 (1981), no. 3, 249–260

\bibitem[H]{HamiltonSingularities} Hamilton, Richard S. ``The formation of singularities in the Ricci flow.'' 
    \textit{Surveys in differential geometry, Vol. II.} 
      (Cambridge, MA, 1993), 7–136, Int. Press, Cambridge, MA, 1995.

\bibitem[I]{Ivey} Ivey, Thomas.
   ``Local existence of Ricci solitons.''
    \textit{Manuscripta Math.} 91 (1996), no. 2, 151--162.
  

\bibitem[KN]{KobayashiNomizu} Kobayashi, Shoshichi; Nomizu, Katsumi.
 \textit{Foundations of differential geometry. Vol. I.}
  Interscience Publishers, a division of John Wiley \& Sons, New York-London, 1963 xi+329 pp.
 
\bibitem[LT]{LuTian} Lu, Peng; Tian, Gang.
``Uniqueness of standard solutions in the work of Perelman.'' Preprint,
{\tt http://math.berkeley.edu/\~{}lott/ricciflow/StanUniqWork2.pdf}.

\bibitem[SW]{ShermanWeinkove}
  Sherman, Morgan; Weinkove, Ben.
  ``Interior derivative estimates for the K¨ahler-Ricci flow.'' Preprint,
  {\tt arXiv:1107.1853v1 [math.DG]}.

\bibitem[S]{Shi} Shi, Wan-Xiong. 
  ``Deforming the metric on complete Riemannian manifolds.''  
  \textit{J. Differential Geom.}  30  (1989),  no. 1, 223--301.


\end{thebibliography}
\end{document}